\documentclass[10pt]{article}

\usepackage[utf8]{inputenc}
\usepackage[T1]{fontenc}

\usepackage{epsf}
\usepackage{amsmath}

\allowdisplaybreaks

\usepackage[showframe=false]{geometry}
\usepackage{changepage}

\usepackage{epsfig}
\usepackage{amssymb}

\usepackage{amsthm}
\usepackage{setspace}
\usepackage{cite}
\usepackage{mcite}

\usepackage{algorithmic}  
\usepackage{algorithm}

\usepackage{shadow}
\usepackage{fancybox}
\usepackage{fancyhdr}

\usepackage{color}
\usepackage[usenames,dvipsnames,svgnames,table]{xcolor}
\newcommand{\bl}[1]{\textcolor{blue}{#1}}

\definecolor{mypurple}{rgb}{.4,.0,.5}

\usepackage[hyphens]{url}

\usepackage[colorlinks=true,
            linkcolor=black,
            urlcolor=blue,
            citecolor=purple]{hyperref}

\usepackage{breakurl}

\def\y{{\bf y}}

\def\x{{\bf x}}

\def\x{{\mathbf x}}

\def\x{{\bf x}}
\def\y{{\bf y}}

\def\be{\begin{equation}}
\def\ee{\end{equation}}
\def\ba{\left[\begin{array}}
\def\ea{\end{array}\right]}

\def\x{{\bf x}}
\def\y{{\bf y}}

\def\1{{\bf 1}}

\def\g{{\bf g}}
\def\0{{\bf 0}}

\def\cG{{\mathcal G}}

\def\mR{{\mathbb R}}

\def\mN{{\mathbb N}}
\def\mE{{\mathbb E}}
\def\mS{{\mathbb S}}

\def\mP{{\mathbb P}}

\def\lp{\left (}
\def\rp{\right )}

\def\y{{\bf y}}

\def\x{{\bf x}}

\def\x{{\mathbf x}}

\def\x{{\bf x}}
\def\y{{\bf y}}

\def\be{\begin{equation}}
\def\ee{\end{equation}}
\def\ba{\left[\begin{array}}
\def\ea{\end{array}\right]}

\def\x{{\bf x}}
\def\y{{\bf y}}

\def\({\left (}
\def\){\right )}

\def\1{{\bf 1}}

\def\g{{\bf g}}
\def\0{{\bf 0}}

\def\cX{{\mathcal X}}

\def\cN{{\mathcal N}}
\def\cX{{\mathcal X}}

\usepackage{xcolor}
\usepackage{color}

\definecolor{darkgreen}{rgb}{0, 0.4,0}

\definecolor{purplebrown}{rgb}{0.5,0.1,0.6}

\definecolor{ultclupcol}{rgb}{0.1,0.5,0.5}

\definecolor{mytrycolor}{rgb}{0.5,0.7,0.2}

\definecolor{ultclupcola}{rgb}{.5,0,.5}

\definecolor{shadebrown}{rgb}{0.1,0.1,0.9}
\definecolor{lightblue}{rgb}{0.2,0,1}

\usepackage{fancybox}
\usepackage{graphicx}
\usepackage{epstopdf}
\usepackage{epsfig}
\usepackage{wrapfig}
\usepackage{subfigure}

\usepackage{xcolor}
\usepackage{tcolorbox}

\newtcbox{\xmybox}{on line,
arc=7pt,
before upper={\rule[-3pt]{0pt}{10pt}},boxrule=0pt,
boxsep=0pt,left=6pt,right=6pt,top=0pt,bottom=0pt,enhanced, coltext=blue, colback=white!10!yellow}

\newtcbox{\xmyboxa}{on line,
arc=7pt,
before upper={\rule[-3pt]{0pt}{10pt}},boxrule=0pt,
boxsep=0pt,left=6pt,right=6pt,top=0pt,bottom=0pt,enhanced, colback=white!10!yellow}

\newtcbox{\xmyboxb}{on line,
arc=7pt,
before upper={\rule[-3pt]{0pt}{10pt}},boxrule=1pt,colframe=darkgreen!100!blue,
boxsep=0pt,left=6pt,right=6pt,top=0pt,bottom=0pt,enhanced, colback=white!10!yellow}

\newtcbox{\xmyboxc}{on line,
arc=7pt,
before upper={\rule[-3pt]{0pt}{10pt}},boxrule=.7pt,colframe=blue!100!blue,
boxsep=0pt,left=6pt,right=6pt,top=0pt,bottom=0pt,enhanced, coltext=blue, colback=white!10!yellow}

\newtcbox{\xmytboxa}{on line,
arc=7pt,
before upper={\rule[-3pt]{0pt}{10pt}},boxrule=.0pt,colframe=pink!50!yellow,
boxsep=0pt,left=6pt,right=6pt,top=0pt,bottom=0pt,enhanced, coltext=white, colback=blue!40!red}

\newtcbox{\xmytboxb}{on line,
arc=7pt,
before upper={\rule[-3pt]{0pt}{10pt}},boxrule=.0pt,colframe=pink!50!yellow,
boxsep=0pt,left=6pt,right=6pt,top=0pt,bottom=0pt,enhanced, coltext=white, colback=white!40!green}

\makeatletter
\newcommand\subsubsubsection{\@startsection{paragraph}{4}{\z@}{-2.5ex\@plus -1ex \@minus -.25ex}{1.25ex \@plus .25ex}{\normalfont\normalsize\bfseries}}
\newcommand\subsubsubsubsection{\@startsection{subparagraph}{5}{\z@}{-2.5ex\@plus -1ex \@minus -.25ex}{1.25ex \@plus .25ex}{\normalfont\normalsize\bfseries}}
\makeatother

\newtheorem{theorem}{Theorem}

\newtheorem{remark}{Remark}

\begin{document}

\begin{singlespace}

\title {An RDT based approach to large deviations of Wishart and Wigner matrices spectral edges  
}
\author{
\textsc{Mihailo Stojnic
\footnote{e-mail: {\tt flatoyer@gmail.com}}
}}
\date{}
\maketitle

\centerline{{\bf Abstract}} \vspace*{0.1in}

We present a novel methodology for studying \emph{large deviations principles} (LDPs) of random matrices. By utilizing a partially lifted variant of \emph{random duality theory} (RDT), we develop a generic LDP framework that completely circumvents traditional random matrix theory (RMT) methods.  To demonstrate the framework's simplicity and accuracy, we apply it to the Wishart and Wigner GOE classical statistical ensembles. In both cases, we obtain elegant LDP characterizations of the upper and lower spectral edges that fully match the results achieved through traditional \emph{Coulomb gas} methodologies in \cite{MajVer09,KatzPer10}.

\vspace*{0.25in} \noindent {\bf Index Terms: Wishart/Wigner matrix; LDP; Random duality theory (RDT)}.

\end{singlespace}

\section{Introduction}
\label{sec:back}

Since its introduction in the first half of the last century \cite{Wishart28}, random matrix theory has been an integral component of various scientific and engineering fields. Its applications range from statistical physics and quantum chromodynamics \cite{Wigner51,Dyson62,FyoKor99,FyoSomm97,Gar88,GarDer88,Verbaa94,Johansson00,AkeKan00} to pure and applied linear algebra \cite{Litvaketal05,Rudelson08}, ODEs \cite{Fyodorov16,Cugli97}, and wireless communications and signal processing/compressed sensing \cite{Fey2008,STS07,CandesTao06,Donoho06CS,StojnicICASSP10var,StojnicISIT2010binary}. Furthermore, RMT is utilized in computational biology \cite{Eisen98}, biological microarrays and genetics \cite{Holter00,ABB00,NovSteph08}, finance \cite{BouPott01}, and statistical machine learning/pattern recognition \cite{BLT20,BMR21,HMRT22,JGH18,Cover65}.

Random matrices and large deviations \cite{DemboZeit98} analyses have also played a crucial role in studying random landscapes. The Kac-Rice method \cite{Kac43,Rice44,Rice45,AdlerTaylor07} is often instrumental for characterizing the \emph{complexity} of critical points and is believed to be fundamentally important in determining the polynomial solvability of random optimizations and the existence of \emph{statistical computational gaps}. More on complexity can be found in statistical physics \cite{Crisanti95,CrisantiComp03,RosBirCam19,Fyod04,FyodWill07,BrayDean07,CavGiaPar98} and mathematics literature \cite{Auff13,Auff13a,BAetal19,HuangSell26,Subag17}. Additionally, details on associated algorithmic aspects, the presence of statistical computational gaps, and their connections with landscape complexities are available in \cite{Stojnicclupsk25,BMPZ23,BarbAKZ23,Subag21,HuangS22a,ElAlMont20,Montanari19,Bald15,Bald20,Stojnicalgbp25}.

\subsection{Relevant prior work}
\label{sec:examples}

This work focuses on two classical Gaussian ensembles: the Wishart \cite{Wishart28} and the Wigner/GOE \cite{Wigner51}. Over the last century, a robust random matrix theory (RMT) has been developed to characterize the various properties of these matrices. Notably, the Wigner and Marchenko-Pastur spectral distribution characterizations \cite{Wigner58,MarPas67} established the foundation for many subsequent developments in the field. Beyond providing a generic treatment of these ensembles, these characterizations highlight the distinct roles of the two primary parts of the spectrum: (i) the \emph{edges} (the endpoints of the spectrum) and (ii) the \emph{bulk} (all other eigenvalues or singular values). Our primary focus is on the spectral edges and their associated \emph{large deviation principles} (LDPs).

\vspace{.1in} 

\noindent \underline{\textbf{\emph{Plain Wigner/GOE/Wishart LDPs}}:} Over the last thirty years a strong theory has been established to study various forms of LDPs, beginning with the characterization of the Wigner matrices' empirical spectral distribution \cite{BenGui97} and the subsequent characterization of the maximal eigenvalue \cite{BenArDemGuio01}. While similar results were achieved for Wishart matrices, both spectral edges must be addressed separately due to the lack of symmetry found in the Wigner/GOE ensemble. Using the \emph{Coulomb gas} methodology \cite{BenArDemGuio01,BenGui97} (for more on Coulomb interactions see also \cite{LiebCol69,Frohlich76,Len61,SerfatyCol24}), the LDPs for the maximal and minimal Wishart eigenvalues were determined in \cite{MajVer09} and \cite{KatzPer10}. Furthermore, research has expanded beyond standard \emph{outer} LDPs to include \emph{inner} LDPs, which characterize deviations toward the interior of the spectral domain. The inner spectral edges LDPs typically have faster rates and have been characterized through similar Coulomb gas methodologies for both Wigner/GOE  \cite{DeanMaj06} and Wishart  \cite{VivMajBoh07,KatzPer10} ensembles.

\vspace{.1in} 

\noindent \underline{\textbf{\emph{Deformations}}:}  A significant body of work exists concerning the deformed variants of the above ensembles. For instance, \cite{GuionnetZeit02} derived the LDP for the empirical measure of full rank deformations of GOE, aligning with predictions from \cite{Matytsin1994}. Additionally, \cite{Maida07} established a largest eigenvalue LDP for rank-one deformed Wigner/GOE matrices. Further research by \cite{Guionnet2012large} examined diagonal deformations of finite rank random matrices with delocalized eigenvectors, determining the LDP of the extreme eigenvalues and matching results for finite rank Wishart matrices found in \cite{Fey2008}. \cite{MckennaDeformed21} obtained extreme eigenvalues LDPs for full rank deformations of Wigner/GOE matrices. Similar results were obtained for diagonally generalized sample covariance matrices in \cite{HussMcKenna24}, building on earlier statistical physics studies \cite{MaillardLDP21, MergnyPott22} (further utilization of replica statistical tools within the context of deformed matrices can also be found in \cite{LeDoussal2025}). Other forms of deformation, such as the sum of random matrices, are explored in \cite{DonatiMaida12, GuionnetMaida20} (for research on products of matrices, see, e.g., \cite{MergnyPott22, RamKatzCas12}). General Jacobi ensembles are discussed in \cite{Forrester12}, while nonsymmetric ensembles are covered in \cite{Byun2026, XuZeng25}. The spectral features of the sums and products can in principle be managed via Voiculescu's free probability theory \cite{Voic91}. However, studying the associated LDPs is much more challenging.

\vspace{.1in} 

\noindent \underline{\textbf{\emph{Impact of different statistics}}:} Many distributional results extend universally \cite{FeldSod10, DeiftGioev07} beyond Gaussian matrices with minimal changes. Conversely, \cite{BordenaveCaputo2014} characterized LDPs of empirical measures for matrices with tails heavier than Gaussians. Similarly, \cite{AugeriNonGauss2016} examined LDPs of the largest eigenvalues for Wigner matrices (see \cite{Groux17} for related results on covariance matrices). In both cases, the LDP speed was found to be slower than in the Gaussian case, with the power of the speed's leading dimensional term corresponding directly to the power of the tail decay.

Analogous eigenvalue problems also arise in studies of graph adjacency matrices. For further details on LDPs and spectral properties in these contexts, please refer to \cite{GangulyNam2022, BhattaSparseGraph2021, CookDembo20, BhattaGang20}, as well as \cite{AugeriBasakSparseWig25} for sparse Wigner matrices.

Regarding Wigner matrices with sub-Gaussian entries, \cite{GuioHuss20} demonstrated that under \emph{sharp sub-Gaussianity}, the resulting LDP matches the Gaussian LDP. However, the broader sub-Gaussian scenario addressed in \cite{Guionnet2012large} can differ from the GOE for very large deviations. Finally, studies with varying variance profiles in \cite{Husson22, Ducatez24} also utilize the sharp sub-Gaussianity concept.

\vspace{.1in} 

\noindent \underline{\textbf{\emph{Beyond LDPs}}:}  Many other properties of random ensembles have been studied in recent literature, yielding several significant results. For example, after the convergence properties of the largest and smallest Wishart eigenvalues were established \cite{BaiLargest88,BaiSmallest93,Silverstein85}, research turned toward their typical moderate deviations or fluctuations. The Tracy-Widom law was subsequently established \cite{TraWid96} (see also \cite{Johnst01,FanJohnsTW22} for different ensembles). From a distributional perspective, these provide more accurate characterizations than LDPs; however, they are restricted to intervals around spectral edges that are not proportional to the extreme eigenvalues.

Regarding the deformed matrices discussed above, \cite{BBP05,Peche06} further examined the properties of their largest eigenvalues, focusing on spectral outliers and the BBP phase transition phenomena. An interesting combination of BBP and LDP is presented in \cite{BirGuioSpike20}, which characterizes the joint LDP of the largest eigenvalue of a rank-one deformed random matrix and the leading eigenvector projection onto the perturbation. Beyond spectral edges, vast literature explores the spectral bulk. Features such as universality, moderate deviations, and eigenvalue spacing across various scenarios—including non-symmetric ones—are now well understood (see, e.g., \cite{JohanssonUniv01,Brezin96,Erdos2012,LeeUnivDefBulk16,TaoNu11,ChenMann94,Cipolloni23,Cipolloni21,Bordenave18}).

While these spectral properties and associated LDPs are largely asymptotic, finite-dimensional considerations are also of interest. A subset of the literature in this direction includes \cite{DumEdel02,Edelman88,Dumitriu08,Koev06,DumEdeShu07,Edelman91,KrishnaiahChang1971}. In many cases, full distributional characterizations for various statistical metrics are obtainable. Because finite dimensions make underlying analyses more challenging, known results typically relate to Gaussian contexts. Attaining simple closed-form characterizations is infeasible, which may preclude direct asymptotic analyses. However, this motivates the search for alternative methods that leverage high-dimensional scenarios, such as the spectral characterizations via empirical measures or moderate and large deviations discussed above.

\subsection{Our contributions}
\label{sec:context}

We develop novel methodology for characterization of random matrices LDPs. As explained above, these problems have already been studied in various context. Typical approaches rely on statistical physics \emph{Coulomb gas} method and so-called \emph{tilted spherical integrations}. We present a completely different approach that circumvents classical random matrix spectral theory and instead relies on \emph{Random duality theory} (RDT) \cite{StojnicRegRndDlt10,StojnicICASSP10var,StojnicISIT2010binary}. Utilizing a particular partially lifted RDT variant, we first derive the LDPs for the upper and lower spectral edges of the Wishart ensemble (Sections \ref{sec:ldpmaxwish} and \ref{sec:ldpminwish}, respectively).  We then complement these results by deriving the LDP for the upper spectral edge of the Wigner/GOE ensemble (Section \ref{sec:ldpmaxwig}; due to the ensemble symmetry, the LDP of the lower spectral edge is identical). We then demonstrate (Sections \ref{sec:ldpwishagrmax} and \ref{sec:ldpwishagrmin}) that our LDP results identically match the Coulomb gas based predictions obtained within statistical physics for the upper and lower Wishart edges in \cite{MajVer09} and \cite{KatzPer10}, respectively. The same type of agreement is demonstrated for the Wigner/GOE matrices as well (Section \ref{sec:ldpwigagrmax}). 

The developed machinery is extremely powerful and can be used to handle LDPs of various different ensembles including many if not all of those studied in prior literature and discussed in the previous section. 
As this is the introductory paper, we chose for the initial presentation two most classical random ensembles, Gaussian Wigner and Gaussian Wishart. Extension to other ensembles rely on the very same concepts introduced here but are technically problem specific. Such technicalities need additional attention and will be discussed elsewhere.

\section{LDPs for Wishart matrices}
\label{sec:ldpwish}

We start by formalizing definitions of mathematical objects that will be used throughout the presentation that follows. For two positive integers $m\in\mN$ and  $n\in\mN$ let matrix $A\in\mR^{m\times n}$ be comprised of independent standard normal entries, i.e., let $A_{ij}\sim \cN(0,1),1\leq i,j\leq n$ with
\begin{equation*}\label{eq:amat1}
  \mE A_{ij}^2=1 \quad \mbox{and}\quad \mE A_{ij}A_{kl}=0, \quad \mbox{for} \quad i\neq k\quad \mbox{and/or}\quad j\neq l.
\end{equation*}
Without loss of generality we assume $m\geq n$ and consider the proportional large dimensional scenarios with
\begin{equation*}\label{eq:amat1a0}
\alpha \triangleq \lim_{n\rightarrow\infty} \frac{m}{n}\geq 1.
\end{equation*}
One then associates with $A$ the following Wishart matrix 
\begin{equation*}\label{eq:amat1a0b0}
W= A^TA.
\end{equation*}
Clearly, $W$ is symmetric and positive semi-definite, i.e., $W=W^T\succeq 0$. In what follows, we will be interested in spectral edges of $W$ and their behavior away from the expected (concentrating) values. To adequately address both spectral edges, we split the presentation in the remaining portion of this section into two subsections: (i) the first one that relates to the right spectral edge (i.e., the maximal eigenvalue) and (ii) the second one that relates to the left spectral edge (i.e., the minimal eigenvalue).

\subsection{Wishart maximal eigenvalue LDP}
\label{sec:ldpmaxwish}

 Let $\mS^n=\{\x|\|\x\|_2=1\}$ be the unit sphere in $\mR^n$ and set 
\begin{eqnarray}\label{eq:inteq1a}
\xi^+_{sph} = \lim_{n\rightarrow\infty} \frac{1}{\sqrt{n}} \max_{\x\in\mS^n,\y\in\mS^m} \y^TA\x.
\end{eqnarray}
It is not that difficult to see that $\xi^+_{sph}$ is both the limiting scaled maximal singular value of $A$ and the square root of the limiting scaled maximal eigenvalue of $W$. In particular, trivial algebraic manipulations give
\begin{eqnarray}\label{eq:inteq1ab0}
\xi^+_{sph} & = & \lim_{n\rightarrow\infty} \frac{1}{\sqrt{n}}\max_{\x\in\mS^n,\y\in\mS^m} \y^TA\x
\nonumber \\
& = &
\lim_{n\rightarrow\infty} \frac{1}{\sqrt{n}} \max_{\x\in\mS^n} \| A\x \|_2
\nonumber \\
& = &
 \lim_{n\rightarrow\infty} \frac{1}{\sqrt{n}} \max_{\x\in\mS^n} \sqrt{\x^TA^TA\x}\nonumber \\
& = &
 \lim_{n\rightarrow\infty} \frac{1}{\sqrt{n}} \sqrt{\lambda_n(A^TA)},
\end{eqnarray}
where function $\lambda_i(\cdot)$ outputs the $i$-th smallest eigenvalue of its argument. As stated above, we assume $\alpha\geq 0$. However, throughout the presentation it will be clear that everything discussed in this subsection can be trivially adjusted to hold for $\alpha\leq 1$.

We are interested in large deviations principles (LDP) associated with $\xi^+_{sph}$. In more concrete terms, we consider the following likelihood
\begin{eqnarray}\label{eq:inteq1ab0}
\mP\lp \xi^+_{sph}\geq u\rp \quad \mbox{ with } \quad u\geq \mE \xi^+_{sph},
\end{eqnarray}
and wonder at what rate it decays to zero as a function of $u$ (a real scalar (basically proportional to $\mE\xi^+_{sph}$)). As stated earlier, this is a classical problem in probability theory that can be attacked via random matrix theory. In that regard, particularly relevant are the results from \cite{MajVer09} where a \emph{Coulomb gas} methodology was utilized and an explicit LDP characterization of $\xi^+_{sph}$ was obtained (see also \cite{HussMcKenna24} for diagonal generalizations and a \emph{tilted spherical integrals} treatment; corresponding statistical physics considerations can be found in \cite{MaillardLDP21,MergnyPott22}). Here, we attack the problem from in a different way. We utilize \emph{Random duality theory} (RDT) (its a partially lifted form) and obtain an alternative LDP formulation. The proposed methodology is much simpler and produces the results summarized in the following theorem.

\begin{theorem}
\label{thm:thm1}
  Consider large $n,m\in\mN$ such that $\lim_{n\rightarrow\infty}\frac{m}{n}\rightarrow\alpha\geq 1$ and let $A_{ij}\sim \cN(0,1)$ be the independent elements of $A\in\mR^{m\times n}$. For a real scalar $c_3\geq 0$ set
  \begin{eqnarray}
   \label{eq:thm1eq1}
 \hat{\gamma}_{1} & = &  \frac{c_3 + \sqrt{c_3^2+4\alpha}}{4}
 \nonumber \\
 \phi_{1}(\gamma_1) & = &  \gamma_{1}c_3 - \frac{\alpha}{2}\log \lp 1 - \frac{c_3}{2 \gamma_{1}}   \rp
 \nonumber \\
 \hat{\gamma}_{2} & = &  \frac{c_3 + \sqrt{c_3^2+4}}{4}
 \nonumber \\
 \phi_{2}(\gamma_2) & = &  \gamma_{2}c_3 - \frac{1}{2}\log \lp 1 - \frac{c_3}{2 \gamma_{2}}   \rp.
 \end{eqnarray}
 Let $\xi^+_{sph}$ be as in (\ref{eq:inteq1ab0}) and for a real scalar $u\geq \mE \xi^+_{sph}$ set
 \begin{eqnarray}
   \label{eq:thm1eq2}
 \nonumber \\
 \phi_u(u) & = & -\frac{c_3^2}{2} + \phi_{1}(\hat{\gamma}_1) + \phi_{2}(\hat{\gamma}_2) -c_3 u .
 \end{eqnarray}
One then has the following LDP with the indicator rate function $\zeta_u(u)$
\begin{eqnarray}
   \label{eq:thm1eq2}
\lim_{n\rightarrow \infty} \frac{\log\lp \mP\lp \xi^+_{sph} \geq u  \rp \rp}{n} \leq \min_{c_3\geq 0}\phi_u(u) \triangleq -\zeta_u(u).
\end{eqnarray}
\end{theorem}

\begin{proof} For standard normal vectors $\g^{(1)}\in\mR^{m\times 1}$, $\g^{(2)}\in\mR^{n\times 1}$,  and standard normal $g\in\mR$ (all with elements independent among themselves and of the elements of $A$), we define two centered Gaussian processes indexed by an array $\cX = \{\x,\y\}$
  \begin{eqnarray}
\label{eq:mr1}
 \cG (\cX) & \triangleq &  \cG (\x,\y)  \triangleq  \sum_{i=1}^n \sum_{j=1}^m A_{i,j}\x_i\y_j  + g  \nonumber   \\
 \cG_u (\cX) & \triangleq &  \cG_u (\x,\y)  \triangleq    \lp\g^{(1)}\rp^T\y +   \lp\g^{(2)}\rp^T\x .
  \end{eqnarray}
Clearly, $\x \in\mR^{n\times 1}$ and $\y \in\mR^{m\times 1}$. Taking two arrays $\cX^{(a_1)}=\{ \x^{(a_1)},\y^{(a_1)}\}$ and $\cX^{(a_2)}=\{ \x^{(a_2)},\y^{(a_2)}\}$ with $\|\x^{(a_i)}\|_2=\|\y^{(a_i)}\|_2=1,i=1,2$, we further write
  \begin{eqnarray}
\label{eq:mr2}
\mE \cG (\cX^{(a_1)})\cG (\cX^{(a_2)})   & =  &    \lp \x^{(a_1)} \rp^T\x^{(a_2)} \lp \y^{(a_1)} \rp^T\y^{(a_2)}  + 1 
\nonumber   \\
\mE \cG_u (\cX^{(a_1)})\cG_u (\cX^{(a_2)})  & =  &  \lp \y^{(a_1)} \rp^T\y^{(a_2)}  +   \lp \x^{(a_1)} \rp^T\x^{(a_2)} .
  \end{eqnarray}
The above effectively evaluates the average correlation/overlap between two system/process replicas for both $\cG (\cX)$ and $\cG_u (\cX)$ processes. From (\ref{eq:mr2}), one further writes
  \begin{eqnarray}
\label{eq:mr5}
  \mE \cG (\cX^{(a_1)})\cG (\cX^{(a_2)})
 -
\mE \cG_u (\cX^{(a_1)})\cG_u (\cX^{(a_2)} )
 & = &    \lp \x^{(a_1)} \rp^T\x^{(a_2)} \lp \y^{(a_1)} \rp^T\y^{(a_2)}  + 1 
\nonumber
\\
& &  - \lp \y^{(a_1)} \rp^T\y^{(a_2)}  -   \lp \x^{(a_1)} \rp^T\x^{(a_2)}
\nonumber
\\
& = &
 \lp 1- \lp \x^{(a_1)} \rp^T\x^{(a_2)}  \rp \lp 1 - \lp \y^{(a_1)} \rp^T\y^{(a_2)}\rp
  \geq   0.
  \end{eqnarray}
We also note that   
  \begin{eqnarray}
\label{eq:mr5a0}
  \mE \cG (\cX^{(a_1)})\cG (\cX^{(a_1)})
 -
\mE \cG_u (\cX^{(a_1)})\cG_u (\cX^{(a_1)} )
  = 
 \lp 1- \lp \x^{(a_1)} \rp^T\x^{(a_1)}  \rp \lp 1 - \lp \y^{(a_1)} \rp^T\y^{(a_1)}\rp
  =   0,
  \end{eqnarray}
  where the last equality follows since $\x^{(a_i)}\in\mS^n$ and/or $\y^{(a_i)}\in\mS^m$ for $i=1,2$.
  
We digress for a moment and recall on Theorem 1.1 from \cite{Gordon85} (the part of the theorem used here is actually known as Slepian lemma and is introduced in \cite{Slep62}).

\begin{theorem}(\cite{Gordon85,Slep62})
\label{thm:Gordonpos1} Let $X_{i}$ and $Y_{i}$, $1\leq i\leq n$, be two centered Gaussian processes which satisfy the following inequalities for all choices of indices
\begin{enumerate}
\item $\mE(X_{i}^2)=\mE(Y_{i}^2)$
\item $\mE(X_{i}X_{l})\leq \mE(Y_{i}Y_{l}), i\neq l$.
\end{enumerate}
 Then
\begin{equation*}
\mE(\min_{i} X_{i})\leq \mE(\min_i Y_{i}) \quad  \Longleftrightarrow \quad \mE(\max_{i} X_{i})\geq \mE(\max_i Y_{i}).
\end{equation*}
\end{theorem}

Applying  Theorem \ref{thm:Gordonpos1} on processes $\cG(\cdot)$ and  $\cG_u(\cdot)$ (with $Y\leftrightarrow\cG$ and $X\leftrightarrow\cG_u$ correspondence) one easily recovers classical Gaussian matrix scaled maximum singular value upper bounds
\begin{equation}\label{eq:mt5a1}
  \xi^+_{sph}\leq \lim_{n\rightarrow\infty} \frac{1}{\sqrt{n}}\lp \sqrt{m} +\sqrt{n}\rp = \sqrt{\alpha} +1.
\end{equation}
Interestingly, although it is stated asymptotically, the bound actually holds for any $m,n\in\mN$. Nonetheless, much more can be done. To that end we also recall on Corollary 1.3 from \cite{Gordon85}.   
\begin{theorem}(\cite{Gordon85})
\label{thm:Gordoncor1} Assume the setup of Theorem \ref{thm:Gordonpos1}. Let $\psi_q(\cdot)$ be an increasing function on the real axis. Then
\begin{equation*}
\mE(\min_{i}\psi_q(X_{i}))\leq \mE(\min_i \psi_q(Y_{i})) \quad \Longleftrightarrow \quad  \mE(\max_{i}\psi_q(X_{i}))\geq \mE(\max_i\psi_q(Y_{i})).
\end{equation*}
\end{theorem}
This result is interesting as it suggest that the above given upper-bound might be tightened. The key question in that regard is whether one can find a better than linear choice for $\psi_q(\cdot)$  (choosing $\psi_q(x)=x$, one clearly has a perfect matching between Theorems \ref{thm:Gordoncor1} and \ref{thm:Gordonpos1}). While it turns out not to be possible (in $n\rightarrow \infty$ regime) to further improve on the bounds given in (\ref{eq:mt5a1}), we discover that choice $\psi_q(x)=e^{c_3 x}$, with $c_3>0$, is particularly interesting. First, we note that Theorems \ref{thm:Gordonpos1} and  \ref{thm:Gordoncor1} are special cases of concepts discussed in Corollary 3 in \cite{Stojnicgscompyx16}  and in Corollary 4 in \cite{Stojnicgscomp16}. Keeping in mind the ideas of \cite{Stojnicgscompyx16,Stojnicgscomp16} and the associated partial RDT lifting \cite{StojnicMoreSophHopBnds10}, one observes that combining (\ref{eq:mr5}) and (\ref{eq:mr5a0}) with the above mentioned choice $\psi_q(x)=e^{c_3 x},c_3>0$,results of Theorem \ref{thm:Gordoncor1}, and correspondence $Y\leftrightarrow\cG$ and $X\leftrightarrow\cG_u$ allows to write
\begin{eqnarray}
\label{eq:mr6}
  \mE \max_{\cX} e^{c_{3}\cG(\cX)} &\leq &   \mE \max_{\cX} e^{c_{3}\cG_u(\cX)},
\end{eqnarray}
where $\max_{\cX}$ means $\max_{\x\in \mS^n,\y\in\mS^m} $. An additional combination of (\ref{eq:mr1}) and (\ref{eq:mr6}) gives
\begin{eqnarray}
\label{eq:mr7}
  \mE \max_{\cX} e^{c_{3}\lp \sum_{i=1}^n \sum_{j=1}^m A_{ij}\x_i\y_j   + g   \rp  } &\leq &   \mE \max_{\cX} e^{c_{3} \lp \lp\g^{(1)}\rp^T\y  + \lp\g^{(2)}\rp^T\x \rp  },
  \end{eqnarray}
which is equivalent to
\begin{eqnarray}
\label{eq:mr8}
   \mE e^{c_{3} \max_{\cX}  \lp \sum_{i=1}^n  \sum_{j=1}^m A_{ij} \x_i\y_j   \rp  } 
 \leq   e^{-\frac{c_{3}^2}{2}} \mE e^{c_{3}  \max_{\cX}  \lp \lp\g^{(1)}\rp^T\y  + \lp\g^{(2)}\rp^T\x \rp  } ,
\end{eqnarray}
and
\begin{eqnarray}
\label{eq:mr9}
     \log \lp \mE e^{c_{3} \max_{\cX}  \lp \sum_{i=1}^n \sum_{j=1}^m A_{ij}\x_i\y_j    \rp  }
   \rp
      \leq 
   -\frac{c_{3}^2}{2} + \log \lp
    \mE e^{c_{3}  \max_{\cX}  \lp  \lp\g^{(1)}\rp^T\y  + \lp\g^{(2)}\rp^T\x   \rp }
    \rp .
\end{eqnarray}
By Chernoff/Markov inequality we have
 \begin{eqnarray}
 \label{eq:ldpeq2}
\frac{1}{n}\log \lp  \mP \lp \xi^+_{sph}
 \geq  u \rp \rp
& \leq  & \frac{1}{n} \log \lp \mE e^{c_3 \sqrt{n} \xi^+_{sph} -c_3\sqrt{n} u} \rp
=  \frac{1}{n} \log \lp \mE e^{c_3 \sqrt{n}  \xi^+_{sph} } \rp -\frac{1}{n}c_3\sqrt{n} u .
  \end{eqnarray}
Combining (\ref{eq:inteq1ab0}), (\ref{eq:mr9}),  and (\ref{eq:ldpeq2}) and adopting scaling $c_3\rightarrow c_3\sqrt{n}$ we find

 \begin{eqnarray}
 \label{eq:ldpeq3}
 \frac{1}{n}\log \lp  \mP \lp \xi^+_{sph}
 \geq  u \rp \rp
& \leq  &    \frac{1}{n} \log \lp \mE e^{c_3 n \xi^+_{sph} } \rp - c_3 u 
\nonumber \\
& =  &    \frac{1}{n} \log \lp \mE e^{c_3 \sqrt{n}  \max_{\cX} \y^TA\x } \rp - c_3 u 
\nonumber \\
& \leq &
 -\frac{c_{3}^2}{2}  +  \frac{1}{n}  \log\lp  \mE_{\g^{(1)},\g^{(2)}} e^{c_{3}\sqrt{n} \max_{\cX}  \lp  \lp\g^{(1)}\rp^T\y  + \lp\g^{(2)}\rp^T\x   \rp } \rp
-c_3 u
\nonumber \\
& = &
  -\frac{c_{3}^2}{2}  +  \frac{1}{n}  \log\lp  \mE_{\g^{(1)}} e^{c_{3}\sqrt{n}\| \g^{(1)}\|_2} \rp
+  \frac{1}{n}  \log\lp  \mE_{\g^{(2)}} e^{c_{3}\sqrt{n}\| \g^{(2)}\|_2} \rp
-c_3 u.
  \end{eqnarray}
Utilizing \cite{StojnicMoreSophHopBnds10}, we further write
 \begin{eqnarray}
 \label{eq:ldpeq4}
 \lim_{n\rightarrow \infty} \frac{1}{n}  \log\lp  \mE_{\g^{(1)}} e^{c_{3}\sqrt{n} \| \g^{(1)}\|_2} \rp
=
\min_{\gamma_1 > 0 } \phi_1(\gamma_1),
\end{eqnarray}
with
 \begin{eqnarray}
 \label{eq:ldpeq4a}
\phi_1(\gamma_1) = \gamma_1 c_3 -\frac{\alpha}{2} \log\lp 1- \frac{c_3}{2\gamma_1}\rp.
  \end{eqnarray}
 Finding the $\gamma_1$ derivative and equaling it to zero gives
 \begin{eqnarray}
 \label{eq:ldpeq4aa0}
\frac{d\phi_1(\gamma)}{d\gamma_1} = c_3 - \frac{\alpha}{ 2\gamma_1 - c_3}  + \frac{\alpha}{2\gamma_1}
=c_3 \lp 1 - \frac{\alpha}{ 2\gamma_1 (2\gamma_1 - c_3 ) } \rp = 0.
  \end{eqnarray}
One then obtains for the optimal $\gamma_1$
 \begin{eqnarray}
 \label{eq:ldpeq5}
\hat{ \gamma }_1 = \frac{c_3  +  \sqrt{c_3^2  +4\alpha} }{4}.
  \end{eqnarray}
Analogously to  (\ref{eq:ldpeq4}), (\ref{eq:ldpeq4a}), and (\ref{eq:ldpeq5}) we then also have
 \begin{eqnarray}
 \label{eq:ldpeq4b0}
 \lim_{n\rightarrow \infty} \frac{1}{n}  \log\lp  \mE_{\g^{(2)}} e^{c_{3}\sqrt{n} \| \g^{(2)}\|_2} \rp
=
\min_{\gamma_2 > 0 } \phi_2(\gamma_2),
\end{eqnarray}
with
 \begin{eqnarray}
 \label{eq:ldpeq4ab1}
\phi_2(\gamma_2) = \gamma_2 c_3 -\frac{1}{2} \log\lp 1- \frac{c_3}{2\gamma_2}\rp.
  \end{eqnarray}
and optimal $\gamma_2$
 \begin{eqnarray}
 \label{eq:ldpeq5b3}
\hat{ \gamma }_2 = \frac{c_3  +  \sqrt{c_3^2  +4} }{4}.
  \end{eqnarray}
A combination of (\ref{eq:ldpeq3})-(\ref{eq:ldpeq5b3}) gives
 \begin{eqnarray}
 \label{eq:ldpeq6}
 \frac{1}{n}\log \lp  \mP \lp \xi^+_{sph}
 \geq  u \rp \rp
& \leq  &  
   -\frac{c_{3}^2}{2}  +  \frac{1}{n}  \log\lp  \mE_{\g^{(1)}} e^{c_{3}\sqrt{n}\| \g^{(1)}\|_2} \rp
+  \frac{1}{n}  \log\lp  \mE_{\g^{(2)}} e^{c_{3}\sqrt{n}\| \g^{(2)}\|_2} \rp
-c_3 u
\nonumber \\
& = & 
   -\frac{c_{3}^2}{2}  + \phi_1(\hat{\gamma}_1)  + \phi_2(\hat{\gamma}_2)  
   -c_3 u .
  \end{eqnarray}
 Optimizing further over $c_3$ gives (\ref{eq:thm1eq2}) and completes the proof of the theorem.
\end{proof}

\subsection{Wishart minimal eigenvalue LDP}
\label{sec:ldpminwish}

Analogously to (\ref{eq:inteq1a}), we set
\begin{eqnarray}\label{eq:mininteq1a}
\xi^-_{sph} = \lim_{n\rightarrow\infty} \frac{1}{\sqrt{n}} \min_{\x\in\mS^n}  \max_{\y\in\mS^m} \y^TA\x.
\end{eqnarray}
and recognize that $\xi^-_{sph}$ is both the limiting scaled minimal singular value of $A$ and the square root of the limiting scaled minimal eigenvalue of $W$. Throughout this section we additionally assume $\alpha>1$ to avoid singularities. Algebraic manipulations similar to those from (\ref{eq:inteq1ab0}) give
\begin{eqnarray}\label{eq:mininteq1ab0}
\xi^-_{sph} & = & \lim_{n\rightarrow\infty} \frac{1}{\sqrt{n}}\min_{\x\in\mS^n} \max_{\y\in\mS^m} \y^TA\x
\nonumber \\
& = &
\lim_{n\rightarrow\infty} \frac{1}{\sqrt{n}} \min_{\x\in\mS^n} \| A\x \|_2
\nonumber \\
& = &
 \lim_{n\rightarrow\infty} \frac{1}{\sqrt{n}} \min_{\x\in\mS^n} \sqrt{\x^TA^TA\x}\nonumber \\
& = &
 \lim_{n\rightarrow\infty} \frac{1}{\sqrt{n}} \sqrt{\lambda_1(A^TA)}.
\end{eqnarray}
As in previous section, we are interested in large deviations principles (LDP) associated with $\xi^-_{sph}$. To that end, we consider the following likelihood
\begin{eqnarray}\label{eq:mininteq1ab0}
\mP\lp \xi^-_{sph}\leq u\rp \quad \mbox{ with } \quad u\leq \mE \xi^-_{sph},
\end{eqnarray}
and aim to determine the rate at which it decays to zero as a function of $u$ ($u$ is again taken as a real scalar proportional to $\mE\xi^-_{sph}$, i.e., independent of $n$). Similarly to the corresponding maximal eigenvalue scenario, determining the rate of probability decay in (\ref{eq:mininteq1ab0})  is also a classical problem in probability theory. In \cite{KatzPer10} a methodology similar to the
\emph{Coulomb gas} approach from \cite{MajVer09}  was developed and utilized to obtain explicit LDP characterization of $\xi^-_{sph}$. Following the approach outlined in Section \ref{sec:ldpmaxwish}, we attack this problem from a different angle and utilize partially lifted (RDT) to obtain an alternative LDP formulation. We summarize the main results that such methodology gives in the following theorem.

\begin{theorem}
\label{thm:thm2}
  Consider large $n,m\in\mN$ such that $\lim_{n\rightarrow\infty}\frac{m}{n}\rightarrow\alpha> 1$ and let $A_{ij}\sim \cN(0,1)$ be the independent elements of $A\in\mR^{m\times n}$. For a real scalar $c_3\geq 0$ set
  \begin{eqnarray}
   \label{eq:minthm1eq1}
 \hat{\gamma}_{1,l} & = &  \frac{-c_3 + \sqrt{c_3^2+4\alpha}}{4}
 \nonumber \\
 \phi_{1,l}(\gamma_{1}) & = &  -\gamma_{1}c_3 - \frac{\alpha}{2}\log \lp 1 + \frac{c_3}{2 \gamma_{1}}   \rp
 \nonumber \\
 \hat{\gamma}_{2} & = &  \frac{c_3 + \sqrt{c_3^2+4}}{4}
 \nonumber \\
 \phi_{2}(\gamma_2) & = &  \gamma_{2}c_3 - \frac{1}{2}\log \lp 1 - \frac{c_3}{2 \gamma_{2}}   \rp.
 \end{eqnarray}
 Let $\xi^-_{sph}$ be as in (\ref{eq:mininteq1ab0}) and for a real scalar $u\leq \mE \xi^-_{sph}$ set
  \begin{eqnarray}
   \label{eq:minthm1eq2}
 \nonumber \\
 \phi_l(u) & = & -\frac{c_3^2}{2} + \phi_{1,l}(\hat{\gamma}_1) + \phi_{2}(\hat{\gamma}_2) +c_3 u .
 \end{eqnarray}
One then has the following LDP with the indicator rate function $\zeta_l(u)$
\begin{eqnarray}
   \label{eq:minthm1eq2}
\lim_{n\rightarrow \infty} \frac{\log\lp \mP\lp \xi^-_{sph} \leq u  \rp \rp}{n} \leq \min_{c_3\geq 0}\phi_l(u)  \triangleq -\zeta_l(u).
\end{eqnarray}
\end{theorem}

\begin{proof} We again consider processes $ \cG (\cX) $ and  $ \cG_u (\cX) $ from (\ref{eq:mr1}) for which forms similar to (\ref{eq:mr2})-(\ref{eq:mr5a0}) will continue to hold. However, we slightly reindex everything. To be  a bit more precise, we keep $\g^{(1)}\in\mR^{m\times 1}$, $\g^{(2)}\in\mR^{n\times 1}$,  and  $g\in\mR$ as in  (\ref{eq:mr1}) and  define two centered Gaussian processes indexed by an array $\cX = \{\x,\y\}$
  \begin{eqnarray}
\label{eq:minmr1}
 \cG (\cX) & \triangleq &  \cG (\x,\y)  \triangleq  \sum_{i=1}^n \sum_{j=1}^m A_{i,j}\x_i\y_j  + g  \nonumber   \\
 \cG_u (\cX) & \triangleq &  \cG_u (\x,\y)  \triangleq    \lp\g^{(1)}\rp^T\y +   \lp\g^{(2)}\rp^T\x .
  \end{eqnarray}
We then take two arrays $\cX^{(a_1,b_1)}=\{ \x^{(a_1)},\y^{(b_1)}\}$ and $\cX^{(a_2,b_2)}=\{ \x^{(a_2)},\y^{(b_2)}\}$ with $\|\x^{(a_i)}\|_2=\|\y^{(b_i)}\|_2=1,i=1,2$ and write
  \begin{eqnarray}
\label{eq:minmr2}
\mE \cG (\cX^{(a_1,b_1)})\cG (\cX^{(a_2,b_2)})   & =  &    \lp \x^{(a_1)} \rp^T\x^{(a_2)} \lp \y^{(b_1)} \rp^T\y^{(b_2)}  + 1 
\nonumber   \\
\mE \cG_u (\cX^{(a_1,b_1)})\cG_u (\cX^{(a_2,b_2)})  & =  &  \lp \y^{(b_1)} \rp^T\y^{(b_2)}  +   \lp \x^{(a_1)} \rp^T\x^{(a_2)} .
  \end{eqnarray}
From (\ref{eq:minmr2}), one then finds
  \begin{eqnarray}
\label{eq:minmr5}
  \mE \cG (\cX^{(a_1,b_1)})\cG (\cX^{(a_2,b_2)})
 -
\mE \cG_u (\cX^{(a_1)})\cG_u (\cX^{(a_2)} )
 & = &    \lp \x^{(a_1)} \rp^T\x^{(a_2)} \lp \y^{(b_1)} \rp^T\y^{(b_2)}  + 1 
\nonumber
\\
& &  - \lp \y^{(b_1)} \rp^T\y^{(b_2)}  -   \lp \x^{(a_1)} \rp^T\x^{(a_2)}
\nonumber
\\
& = &
 \lp 1- \lp \x^{(a_1)} \rp^T\x^{(a_2)}  \rp \lp 1 - \lp \y^{(b_1)} \rp^T\y^{(b_2)}\rp
\nonumber
\\
& \geq &   0.
  \end{eqnarray}
We also note that   
  \begin{equation}
\label{eq:minmr5a0}
  \mE \cG (\cX^{(a_1,b_1)})\cG (\cX^{(a_1,b_2)})
 -
\mE \cG_u (\cX^{(a_1,b_1)})\cG_u (\cX^{(a_1,b_2)} )
  = 
 \lp 1- \lp \x^{(a_1)} \rp^T\x^{(a_1)}  \rp \lp 1 - \lp \y^{(b_1)} \rp^T\y^{(b_2)}\rp
  =   0 \leq 0,
  \end{equation}
  where the last equality follows since $\x^{(a_i)}\in\mS^n$, $i=1,2$. Moreover, one also has
  \begin{equation}
\label{eq:minmr5a0a0}
  \mE \cG (\cX^{(a_1,b_1)})\cG (\cX^{(a_1,b_1)})
 -
\mE \cG_u (\cX^{(a_1,b_1)})\cG_u (\cX^{(a_1,b_1)} )
  = 
 \lp 1- \lp \x^{(a_1)} \rp^T\x^{(a_1)}  \rp \lp 1 - \lp \y^{(b_1)} \rp^T\y^{(b_1)}\rp
  =   0.
  \end{equation}

We recall on the following (complete) version of Theorem 1.1 from \cite{Gordon85}.  
\begin{theorem}(\cite{Gordon85})
\label{thm:minGordonpos1} Let $X_{ij}$ and $Y_{ij}$, $1\leq i\leq n$, $1\leq j\leq m$, be two centered Gaussian processes which satisfy the following inequalities for all choices of indices
\begin{enumerate}
\item $\mE(X_{ii}^2)=\mE(Y_{ii}^2)$
\item $\mE(X_{ij}X_{il})\geq \mE(Y_{ij}Y_{il})$
\item $\mE(X_{ij}X_{lk})\leq \mE(Y_{ij}Y_{lk})$, $i\neq l$.
\end{enumerate}
 Then
\begin{equation*}
\mE(\min_{i} \max_{j} X_{ij})\leq \mE(\min_i \max_{j} Y_{ij}) \quad  \Longleftrightarrow \quad \mE(\max_{i}\min_{j} X_{ij})\geq \mE(\max_i \min_{j} Y_{ij}).
\end{equation*}
\end{theorem}

\begin{remark}
  One should note that while Theorem \ref{thm:Gordonpos1} is technically a special case of Theorem  \ref{thm:minGordonpos1}, it is known as Slepian lemma \cite{ Slep62} and it existed on its own long before Theorem  \ref{thm:minGordonpos1} was introduced in \cite{Gordon85}. For further developments and even more general variants of which Theorem  \ref{thm:minGordonpos1} is a special case see, e.g., \cite{Stojnicgscompyx16,Stojnicgscomp16}.
\end{remark}

Applying  Theorem \ref{thm:minGordonpos1} on processes $\cG(\cdot)$ and  $\cG_u(\cdot)$ (with $Y\leftrightarrow\cG$ and $X\leftrightarrow\cG_u$ correspondence) one recovers another classical result related to the minimum singular/eigenvalue of Wishart matrices
\begin{equation}\label{eq:minmt5a1}
  \xi^-_{sph}\geq \lim_{n\rightarrow\infty} \frac{1}{\sqrt{n}}\lp \sqrt{m} -\sqrt{n}\rp = \sqrt{\alpha} -1.
\end{equation}
We again note that even though everything is stated asymptotically, the bound actually holds for any $m,n\in\mN$. Similarly to what was observed in the previous section, much more actually can be done. To that end we recall on the following (complete) Corollary 1.3 from \cite{Gordon85} (see also \cite{StojnicMoreSophHopBnds10} for a slight refinement needed below).
\begin{theorem}(\cite{Gordon85})
\label{thm:minGordoncor1} Assume the setup of Theorem \ref{thm:minGordonpos1}. Let $\psi_q(\cdot)$ be an increasing function on the real axis. Then
\begin{equation*}
\mE(\min_{i}\max_{j}\psi_q(X_{ij}))\leq \mE(\min_i\max_{j} \psi_q(Y_{ij})).
\end{equation*}
Also, let $\psi_q(\cdot)$ be a decreasing function on the real axis. Then
\begin{equation*}
   \mE(\max_{i}\min_{j}\psi_q(X_{ij}))\geq \mE(\max_i\min_{j}\psi_q(Y_{ij})).
\end{equation*}
\end{theorem}

Following into the footsteps of the mechanism presented in Section \ref{sec:ldpmaxwish}, we  observe that choosing  $\psi_q(x)=e^{-c_3 x},c_3>0$ and combining such a choice with (\ref{eq:minmr5}) and (\ref{eq:minmr5a0}), results of Theorem \ref{thm:minGordoncor1}, and $Y\leftrightarrow\cG$ and $X\leftrightarrow\cG_u$ correspondence gives
\begin{eqnarray}
\label{eq:minmr6}
  \mE \max_{\x\in\mS^n} \min_{\y\in\mS^m} e^{-c_{3}\cG(\x,\y)} &\leq &   \mE \max_{\x\in\mS^n} \min_{\y\in\mS^m}  e^{-c_{3}\cG_u(\x,\y)} .
\end{eqnarray}
Combining further (\ref{eq:minmr1}) and (\ref{eq:minmr6}), one finds
\begin{eqnarray}
\label{eq:minmr7}
  \mE \max_{\x\in\mS^n} \min_{\y\in\mS^m}  e^{-c_{3}\lp \sum_{i=1}^n \sum_{j=1}^m A_{ij}\x_i\y_j   + g   \rp  } &\leq &   \mE \max_{\x\in\mS^n} \min_{\y\in\mS^m}  e^{-c_{3} \lp \lp\g^{(1)}\rp^T\y  + \lp\g^{(2)}\rp^T\x \rp  },
  \end{eqnarray}
which is equivalent to
\begin{eqnarray}
\label{eq:minmr8}
   \mE e^{c_{3} \max_{\x\in\mS^n} \min_{\y\in\mS^m}   -\lp \sum_{i=1}^n  \sum_{j=1}^m A_{ij} \x_i\y_j   \rp  } 
 \leq   e^{-\frac{c_{3}^2}{2}} \mE e^{  c_{3}  \max_{\x\in\mS^n} \min_{\y\in\mS^m}   -\lp \lp\g^{(1)}\rp^T\y  + \lp\g^{(2)}\rp^T\x \rp  } ,
\end{eqnarray}
and
\begin{equation}
\label{eq:minmr9}
     \log \lp \mE e^{ c_{3} \max_{\x\in\mS^n} \min_{\y\in\mS^m}  -\lp \sum_{i=1}^n \sum_{j=1}^m A_{ij}\x_i\y_j    \rp  }
   \rp
      \leq 
   -\frac{c_{3}^2}{2} + \log \lp
    \mE e^{c_{3}  \max_{\x\in\mS^n} \min_{\y\in\mS^m}  -\lp  \lp\g^{(1)}\rp^T\y  + \lp\g^{(2)}\rp^T\x   \rp }
    \rp.
\end{equation}
Relying again on Chernoff/Markov inequality, we further write
 \begin{eqnarray}
 \label{eq:minldpeq2}
\frac{1}{n}\log \lp  \mP \lp \xi^-_{sph}
 \leq  u \rp \rp
& \leq  & \frac{1}{n} \log \lp \mE e^{-c_3 \sqrt{n} \xi^-_{sph} +c_3\sqrt{n} u} \rp
=  \frac{1}{n} \log \lp \mE e^{-c_3 \sqrt{n}  \xi^-_{sph} } \rp +\frac{1}{n}c_3\sqrt{n} u .
  \end{eqnarray}
A combination of (\ref{eq:inteq1ab0}), (\ref{eq:mr9}),  and (\ref{eq:ldpeq2}) together with scaling $c_3\rightarrow c_3\sqrt{n}$ gives

 \begin{eqnarray}
 \label{eq:minldpeq3}
 \frac{1}{n}\log \lp  \mP \lp \xi^-_{sph}
 \leq  u \rp \rp
& \leq  &    \frac{1}{n} \log \lp \mE e^{-c_3 n \xi^-_{sph} } \rp + c_3 u 
\nonumber \\
& =  &    \frac{1}{n} \log \lp \mE e^{-c_3 \sqrt{n} \min_{\x\in\mS^n} \max_{\y\in\mS^m} \y^TA\x } \rp + c_3 u 
\nonumber \\
& =  &    \frac{1}{n} \log \lp \mE e^{c_3 \sqrt{n} \max_{\x\in\mS^n} \min_{\y\in\mS^m} -\y^TA\x } \rp + c_3 u 
\nonumber \\
& \leq &
 -\frac{c_{3}^2}{2}  +  \frac{1}{n}  \log\lp  \mE_{\g^{(1)},\g^{(2)}} e^{c_{3}\sqrt{n} \max_{\x\in\mS^n} \min_{\y\in\mS^m} -  \lp  \lp\g^{(1)}\rp^T\y  + \lp\g^{(2)}\rp^T\x   \rp } \rp
+c_3 u
\nonumber \\
& = &
  -\frac{c_{3}^2}{2}  +  \frac{1}{n}  \log\lp  \mE_{\g^{(1)}} e^{-c_{3}\sqrt{n}\| \g^{(1)}\|_2} \rp
+  \frac{1}{n}  \log\lp  \mE_{\g^{(2)}} e^{c_{3}\sqrt{n}\| \g^{(2)}\|_2} \rp
+c_3 u.
  \end{eqnarray}
Following \cite{StojnicMoreSophHopBnds10} we write
 \begin{eqnarray}
 \label{eq:minldpeq4}
 \lim_{n\rightarrow \infty} \frac{1}{n}  \log\lp  \mE_{\g^{(1)}} e^{-c_{3}\sqrt{n} \| \g^{(1)}\|_2} \rp
=
\max_{\gamma_1 > 0 } \phi_{1,l}(\gamma_1),
\end{eqnarray}
with
 \begin{eqnarray}
 \label{eq:minldpeq4a}
\phi_{1,l}(\gamma_1) = -\gamma_1 c_3 -\frac{\alpha}{2} \log\lp 1 + \frac{c_3}{2\gamma_1}\rp.
  \end{eqnarray}
 Finding the $\gamma_1$ derivative and equaling it to zero gives
 \begin{eqnarray}
 \label{eq:minldpeq4aa0}
\frac{d\phi_{1,l}(\gamma)}{d\gamma_1} = -c_3 - \frac{\alpha}{ 2\gamma_1 + c_3}  + \frac{\alpha}{2\gamma_1}
=c_3 \lp -1 + \frac{\alpha}{ 2\gamma_1 (2\gamma_1 + c_3 ) } \rp = 0.
  \end{eqnarray}
One then obtains for the optimal $\gamma_1$
 \begin{eqnarray}
 \label{eq:minldpeq5}
\hat{ \gamma }_{1,l} = \frac{-c_3  +  \sqrt{c_3^2  +4\alpha} }{4}.
  \end{eqnarray}
From  (\ref{eq:ldpeq4b0})-(\ref{eq:ldpeq5b3}), we also have
 \begin{eqnarray}
 \label{eq:minldpeq4b0}
 \lim_{n\rightarrow \infty} \frac{1}{n}  \log\lp  \mE_{\g^{(2)}} e^{c_{3}\sqrt{n} \| \g^{(2)}\|_2} \rp
=
  \phi_2(\hat{\gamma}_2),
\end{eqnarray}
with
 \begin{eqnarray}
 \label{eq:minldpeq5b3}
\hat{ \gamma }_2 = \frac{c_3  +  \sqrt{c_3^2  +4} }{4}.
  \end{eqnarray}
A combination of (\ref{eq:minldpeq3})-(\ref{eq:minldpeq5b3}) then gives
 \begin{eqnarray}
 \label{eq:minldpeq6}
 \frac{1}{n}\log \lp  \mP \lp \xi^-_{sph}
 \leq  u \rp \rp
& \leq  &  
   -\frac{c_{3}^2}{2}  +  \frac{1}{n}  \log\lp  \mE_{\g^{(1)}} e^{-c_{3}\sqrt{n}\| \g^{(1)}\|_2} \rp
+  \frac{1}{n}  \log\lp  \mE_{\g^{(2)}} e^{c_{3}\sqrt{n}\| \g^{(2)}\|_2} \rp
+c_3 u
\nonumber \\
& = & 
   -\frac{c_{3}^2}{2}  + \phi_{1,l}(\hat{\gamma}_1)  + \phi_2(\hat{\gamma}_2)  
   +c_3 u .
  \end{eqnarray}
Additional  optimization over $c_3$ gives (\ref{eq:minthm1eq2}) and completes the proof of the theorem.
\end{proof}

\subsection[]{Agreement with results from \cite{MajVer09,KatzPer10}}
\label{sec:ldpwishagr}
 
 We first consider the upper spectral edge, i.e., the maximal eigenvalue and then the lower spectral edge and the minimal eigenvalue.

\subsubsection{Upper spectral edge}
\label{sec:ldpwishagrmax}
 
 As stated earlier, in  \cite{MajVer09} the maximal Wishart eigenvalue LDP was obtained via a Coulomb gas approach. The following is the explicit characterization obtained in \cite{MajVer09}. 
 
\begin{eqnarray}
\label{eq:agr1}
 x & = &  u^2 - (\sqrt{\alpha}+1)^2, \quad  
  b = \lp \sqrt{\alpha} +1 \rp^2, \quad
   a = \lp \sqrt{\alpha} -1 \rp^2
  \nonumber 
  \\
f(x') & = & \frac{1}{2\pi x} \sqrt{ ( b - x' )(x' - a) } ;
  \nonumber 
  \\
I_{\Phi} & = & \int_{a}^{b} \log\lp \frac{x+b-x'}{b-x'}  \rp f(x')  dx' 
 \nonumber  \\
 \Phi_u & = & \frac{x}{2} - \frac{\alpha-1}{2}\log\lp \frac{x+b}{b}\rp -I_{\Phi} .
  \end{eqnarray} 

In Figure \ref{fig:fig1} (the right portion) we show in parallel both $\zeta_u$ obtained via partially lifted RDT in Section \ref{sec:ldpmaxwish} and  $\Phi_u$ obtained via Coulomb gas approach in  \cite{MajVer09}  and given in (\ref{eq:agr1}). We choose $\alpha=4$ and consider the $u$ values outside the bulk of the spectrum and above the upper spectral edge, i.e., we consider the $u$ values such that $u\geq \mE \xi^+_{sph} = \sqrt{\alpha} +1 = 3$. The obtained $\zeta_u$ and $\Phi_u$ curves are visually identical, i.e., the proposed RDT mechanism fully matches the results obtained via Coulomb gas approach.

  \begin{figure}[h]
\centering
\centerline{\includegraphics[width=.87\linewidth]{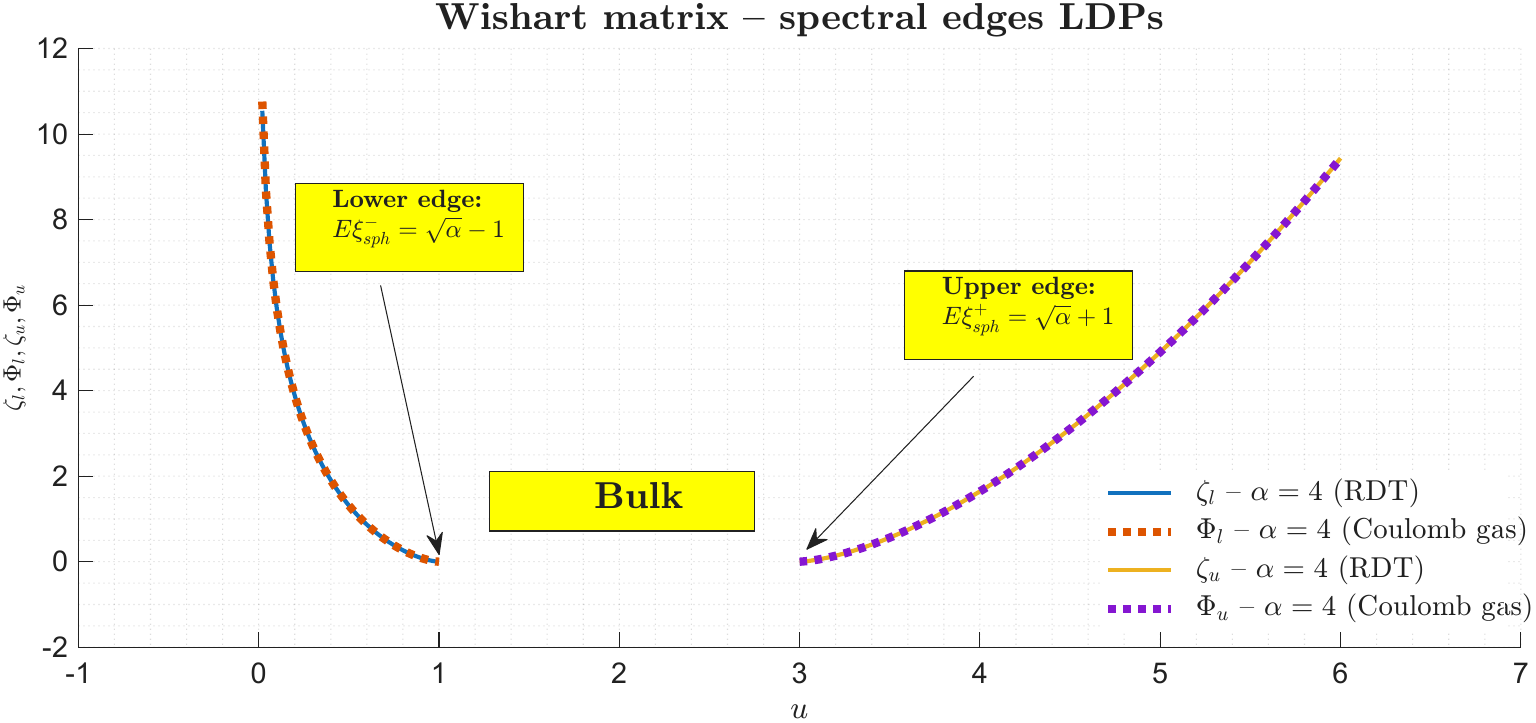}}
\caption{Wishart matrix -- lower ($\zeta_l$, $\Phi_l$) and upper ($\zeta_u$, $\Phi_u$) spectral edges LDPs; $\alpha=4$; RDT ($\zeta_l$, $\zeta_u$)  versus Coulomb gas ($\Phi_l$, $\Phi_u$) }
\label{fig:fig1}
\end{figure}

To give a bit more flavor as to what kind of agreement between the two methodologies is achieved we complement Figure \ref{fig:fig1} with Table \ref{tab:tab1}. For three concrete $u$ values ($u\in\{4,5,6\}$), both $\zeta_u$ and $\Phi_u$ are given. We also accompany these with all relevant RDT parameters ($\hat{c}_3$ is obtained as the solution of the optimization in (\ref{eq:thm1eq2}) whereas $\hat{\gamma}_1$ and $\hat{\gamma}_2$ are evaluated for $c_3=\hat{c}_3$). As the values given in  Table \ref{tab:tab1} show, one has a full numerical agrement between $\zeta_u$ and $\Phi_u$.

\begin{table}[h]
\caption{Wishart upper spectral edge LDP rate function values obtained via RDT ($\zeta_u$) and via Coulomb gas ($\Phi_u$) ; $\alpha=4$ }\vspace{.1in}
\centering
\def\arraystretch{1.2}
\begin{tabular}{ ||c||c||c||c|| }\hline\hline
 \hspace{-0in}  $u$                                             &  $4$ &  $5$ & $6$   \\ \hline\hline
    $\hat{c}_3$                 &   $ 2.5617 $ &   $ 3.9192$   &   $ 5.1235 $   \\ \hline
$\hat{\gamma}_1$                                         &   $ 1.8279$ &   $ 2.3798$   &   $2.9059 $   \\ \hline
      $\hat{\gamma}_2$                                      &   $ 1.4529 $ &   $ 2.0798 $   &  $  2.6559 $   \\ \hline \hline 
$\zeta_u(u)$                                         &  \bl{$\mathbf{1.6440}$} &  \bl{$\mathbf{ 4.9045}$}   &  \bl{$\mathbf{ 9.4336 }$}   \\ \hline \hline
$\Phi_u(u)$                                         &  $\mathbf{1.6440}$ &  $\mathbf{ 4.9045}$   &  $\mathbf{ 9.4336  }$   \\ 
\hline\hline
\end{tabular}
\label{tab:tab1}
\end{table}

\subsubsection{Lower spectral edge}
\label{sec:ldpwishagrmin}
 
 As stated earlier, in   \cite{KatzPer10}
 the minimal Wishart eigenvalue LDP was obtained through an extension of a Coulomb gas methodology from  \cite{MajVer09}. The following explicit rate function characterization was obtained in \cite{KatzPer10}. 
 
\begin{eqnarray}
\label{eq:minagr1}
 x & = &  (\sqrt{\alpha}-1)^2 - u^2, \quad  
  b = \lp \sqrt{\alpha} +1 \rp^2, \quad
   a = \lp \sqrt{\alpha} -1 \rp^2 
  \nonumber 
  \\
I_{\Phi,1} & = & 2\log\lp \frac{\sqrt{x+b-a}-\sqrt{x}}{\sqrt{b-a}} \rp  
 \nonumber  \\
I_{\Phi,2} & = & (\alpha-1) \log\lp 1+ 2\sqrt{\frac{x}{a}}\frac{(\sqrt{x+b-a}-\sqrt{x})}{b-a} \rp  
 \nonumber  \\
 \Phi_l & = &  - \frac{\alpha-1}{2}\log\lp 1-\frac{x}{a} \rp   - \frac{1}{2}\sqrt{x}\sqrt{x+b-a} + I_{\Phi,1} + I_{\Phi,2} .
  \end{eqnarray}

In Figure \ref{fig:fig1} (the left portion) we show in parallel both $\zeta_l$ obtained via partially lifted RDT in Section \ref{sec:ldpminwish} and  $\Phi_l$ obtained via Coulomb gas approach  in  \cite{KatzPer10} and summarized in (\ref{eq:minagr1}). We again choose $\alpha=4$ and consider the $u$ values outside the bulk of the spectrum and below the lower spectral edge, i.e., we consider the $u$ values such that $u\leq \mE \xi^-_{sph} = \sqrt{\alpha} -1 = 1$. As Figure \ref{fig:fig1} shows, the obtained $\zeta_l$ and $\Phi_l$ curves are again visually indistinguishable implying that the proposed RDT mechanism fully matches the results obtained via the Coulomb gas approach.

To ensure that the above agreement is not only visual but also fully numerical as well, we complement Figure \ref{fig:fig1} with Table \ref{tab:tab2}. We again take three concrete $u$ values ($u\in\{0.2,0.5,0.8\}$), and evaluate both $\zeta_l$ and $\Phi_l$. All relevant RDT parameters ($\hat{c}_3$ obtained as the solution of the optimization in (\ref{eq:minthm1eq2}) as well as  $\hat{\gamma}_{1,l}$ and $\hat{\gamma}_2$ evaluated for $c_3=\hat{c}_3$) are provided in Table \ref{tab:tab2} as well. We obsrve a full numerical agrement between $\zeta_l$ and $\Phi_l$.

\begin{table}[h]
\caption{Wishart lower spectral edge LDP rate function values obtained via RDT ($\zeta_l$) and via Coulomb gas ($\Phi_l$) ; $\alpha=4$ }\vspace{.1in}
\centering
\def\arraystretch{1.2}
\begin{tabular}{ ||c||c||c||c|| }\hline\hline
 \hspace{-0in}  $u$                                             &  $0.2$ &  $0.5$ & $0.8$   \\ \hline\hline
 $\hat{c}_3$                &   $ 14.6642 $ &   $ 5.1235 $   &   $  2.1685 $   \\ \hline
$\hat{\gamma}_1$                                         &   $  0.1339$ &   $ 0.3441$   &   $ 0.5954 $   \\ \hline
  $\hat{\gamma}_2$                                         &   $  7.3661 $ &   $ 2.6559 $   &  $  1.2796 $   \\ \hline \hline 
$\zeta_l(u)$                                         &  \bl{$\mathbf{3.8850 }$} &  \bl{$\mathbf{ 1.3161 }$}   &  \bl{$\mathbf{  0.2672 }$}   \\ \hline \hline
$\Phi_l(u)$                                         &  $\mathbf{3.8850}$ &  $\mathbf{  1.3161 }$   &  $\mathbf{  0.2672  }$   \\ 
\hline\hline
\end{tabular}
\label{tab:tab2}
\end{table}

\section{Wigner matrix LDPs}
\label{sec:ldpwig}

We now consider plain Gaussian (Ginibre) random matrix $A$ \cite{Ginibre65}. As these matrices are typical examples where Wigner properties (featured in random matrices with independent and identically distributed entries)  are in place, we refer to their symmetrized variants as (Gaussian) Wigner (or Wigner/GOE) matrices. The following are the symmetric variants of interest
\begin{equation*}\label{eq:wigamat1a0b0}
\bar{W}= \frac{1}{2} \lp A^T + A \rp.
\end{equation*}
As will be clear later on, instead of more typical $\frac{1}{\sqrt{2}}$ scaling, we adopt $\frac{1}{2}$ so that the results are in a direct agreement with those from \cite{MajVer09}. Also, due to the symmetry of the centered Gaussians, it is sufficient to consider only one side of $\bar{W}$'s spectrum. We focus on the maximal eigenvalue of $\bar{W}$. To make the presentation easier to follow, we parallel the exposition of Section \ref{sec:ldpmaxwish}. However, we proceed in a slightly faster fashion, avoid unnecessary repetitions, and instead focus on key differences.

\subsection{$\bar{W}$'s maximal eigenvalue LDP}
\label{sec:ldpmaxwig}

For the unit sphere  $\mS^n=\{\x|\|\x\|_2=1\}$ we set 
\begin{eqnarray}\label{eq:wiginteq1a}
\xi^+_{wig} = \lim_{n\rightarrow\infty} \frac{1}{\sqrt{n}} \max_{\x\in\mS^n} \x^T\bar{W}\x.
\end{eqnarray}
It immediately follows that $\xi^+_{wig}$ is the limiting scaled maximal eigenvalue of $\bar{W}$. Simple algebraic manipulations also give
\begin{eqnarray}\label{eq:wiginteq1ab0}
\xi^+_{wig} 
& = &
 \lim_{n\rightarrow\infty} \frac{1}{\sqrt{n}} \sqrt{\lambda_n(\bar{W})},
 \nonumber \\
& = & \lim_{n\rightarrow\infty} \frac{1}{\sqrt{n}}\max_{\x\in\mS^n} \x^T\bar{W}\x
\nonumber \\
& = &
\lim_{n\rightarrow\infty} \frac{1}{\sqrt{n}}  \max_{\x\in\mS^n} \frac{1}{2}\x^T(A^T+A)\x
\nonumber \\
& = &
 \lim_{n\rightarrow\infty} \frac{1}{\sqrt{n}} \max_{\x\in\mS^n} \x^TA\x,
\end{eqnarray}
where, as earlier, function $\lambda_i(\cdot)$ outputs the $i$-th smallest eigenvalue of its argument. 

In what follows we are interested in LDPs associated with $\xi^+_{wig}$, i.e., we look at the rate of decay (as a function of $u$ proportional to $\mE \xi^+_{wig}$) of the likelihood
\begin{eqnarray}\label{eq:wiginteq1ab1}
\mP\lp \xi^+_{wig}\geq u\rp \quad \mbox{ with } \quad u\geq \mE \xi^+_{wig}.
\end{eqnarray}
  As for the Wishart matrices, this is also a classical problem in probability theory. It was attacked in \cite{MajVer09} via the very same \emph{Coulomb gas} methodology utilized for characterization of Wishart matrices LDPs. As a result, explicit LDP rate function was derived. We below show that one can again utilize partially lifted RDT and obtain an alternative matching LDP formulation. The obtained results are summarized in the following theorem.

\begin{theorem}
\label{thm:wigthm1}
  Consider large $n\in\mN$  and let $A_{ij}\sim \cN(0,1)$ be the independent elements of $A\in\mR^{n\times n}$. For a real scalar $c_3\geq 0$ set
  \begin{eqnarray}
   \label{eq:wigthm1eq1}
 \hat{\bar{\gamma}}_{1} & = &  \frac{c_3\sqrt{2} + \sqrt{2 c_3^2+4}}{4}
 \nonumber \\
 \bar{\phi}_{1}(\gamma_1) & = &  \gamma_{1}c_3\sqrt{2} - \frac{1}{2}\log \lp 1 - \frac{c_3\sqrt{2}}{2 \gamma_{1}}   \rp
 \end{eqnarray}
 Let $\xi^+_{wig}$ be as in (\ref{eq:wiginteq1ab0}) and for a real scalar $u\geq \mE \xi^+_{wig}$ set
 \begin{eqnarray}
   \label{eq:wigthm1eq2}
 \nonumber \\
 \bar{\phi}_u(u) & = & -\frac{c_3^2}{2} + \bar{\phi}_{1}(\hat{\bar{\gamma}}_1)  -c_3 u .
 \end{eqnarray}
One then has the following LDP with the indicator rate function $\zeta_u(u)$
\begin{eqnarray}
   \label{eq:wigthm1eq2}
\lim_{n\rightarrow \infty} \frac{\log\lp \mP\lp \xi^+_{wig} \geq u  \rp \rp}{n} \leq \min_{c_3\geq 0}\bar{\phi}_u(u) \triangleq -\bar{\zeta}_u(u).
\end{eqnarray}
\end{theorem}

\begin{proof} As hinted above, we proceed by paralleling the proof of Theorem \ref{thm:thm1}. For standard normal vector $\g^{(1)}\in\mR^{n\times 1}$ and scalar $g\in\mR$ (all with elements independent among themselves and of the elements of $A$), we define two centered Gaussian processes indexed by an array $\cX = \{\x\}$
  \begin{eqnarray}
\label{eq:wigmr1}
 \bar{\cG} (\cX) & \triangleq &  \bar{\cG} (\x)  \triangleq  \sum_{i=1}^n \sum_{j=1}^m A_{i,j}\x_i\x_j  + g  \nonumber   \\
 \bar{\cG}_u (\cX) & \triangleq &  \bar{\cG}_u (\x)  \triangleq    \sqrt{2}\lp\g^{(1)}\rp^T\x .
  \end{eqnarray}
Taking two arrays $\cX^{(1)}=\{ \x^{(1)}\}$ and $\cX^{(2)}=\{ \x^{(2)}\}$ with $\|\x^{(i)}\|_2=1,i=1,2$, we have
  \begin{eqnarray}
\label{eq:wigmr2}
\mE \bar{\cG} (\cX^{(1)})\bar{\cG} (\cX^{(2)})   & =  &    \lp \lp \x^{(1)} \rp^T\x^{(2)} \rp^2    + 1 
\nonumber   \\
\mE \bar{\cG}_u (\cX^{(1)})\bar{\cG}_u (\cX^{(2)})  & =  &     2 \lp \x^{(1)} \rp^T\x^{(2)} .
  \end{eqnarray}
From (\ref{eq:wigmr2}), one further finds
  \begin{eqnarray}
\label{eq:wigmr5}
  \mE \bar{\cG} (\cX^{(1)}) \bar{\cG} (\cX^{(2)})
 -
\mE \bar{\cG}_u (\cX^{(1)}) \bar{\cG}_u (\cX^{(2)} )
 & = &   \lp \lp \x^{(1)} \rp^T\x^{(2)} \rp^2  + 1 
      -  2 \lp \x^{(1)} \rp^T\x^{(2)}
\nonumber
\\
& = &
 \lp 1- \lp \x^{(1)} \rp^T\x^{(2)}  \rp^2
  \geq   0.
  \end{eqnarray}
We additionally note    
  \begin{eqnarray}
\label{eq:wigmr5a0}
  \mE \bar{\cG} (\cX^{(1)}) \bar{\cG} (\cX^{(1)})
 -
\mE \bar{\cG}_u (\cX^{(1)}) \bar{\cG}_u (\cX^{(1)} )
  = 
 \lp 1- \lp \x^{(a_1)} \rp^T\x^{(a_1)}  \rp^2  
  =   0,
  \end{eqnarray}
  where the last equality follows since $\x^{(i)}\in\mS^n$ for $i=1,2$. Applying  Theorem \ref{thm:Gordonpos1} on processes $\bar{\cG}(\cdot)$ and  $\bar{\cG}_u(\cdot)$ (with $Y\leftrightarrow \bar{\cG}$ and $X\leftrightarrow \bar{\cG}_u$ correspondence) one easily recovers classical Gaussian/Wigner matrix scaled maximum eigenvalue upper bound
\begin{equation}\label{eq:wigmt5a1}
  \xi^+_{wig}\leq \lim_{n\rightarrow\infty} \frac{1}{\sqrt{n}}\lp \sqrt{2}\sqrt{n}\rp = \sqrt{2}.
\end{equation}
  
  Keeping in mind once again the ideas of \cite{Stojnicgscompyx16,Stojnicgscomp16} and the partial RDT lifting mechanism from \cite{StojnicMoreSophHopBnds10}, one observes that combining (\ref{eq:wigmr5}) and (\ref{eq:wigmr5a0}) with results of Theorem \ref{thm:Gordoncor1} and choosing $\psi_q(x)=e^{c_3 x},c_3>0$ gives
\begin{eqnarray}
\label{eq:wigmr6}
  \mE \max_{\cX} e^{c_{3}\bar{\cG}(\cX)} &\leq &   \mE \max_{\cX} e^{c_{3}\bar{\cG}_u(\cX)},
\end{eqnarray}
where $\max_{\cX}$ now stands for $\max_{\x\in \mS^n} $. Combining further  (\ref{eq:wigmr1}) and (\ref{eq:wigmr6}) gives
\begin{eqnarray}
\label{eq:wigmr7}
  \mE \max_{\cX} e^{c_{3}\lp \sum_{i=1}^n \sum_{j=1}^m A_{ij}\x_i\y_j   + g   \rp  } &\leq &   \mE \max_{\cX} e^{c_{3} \sqrt{2}\lp\g^{(1)}\rp^T\x   },
  \end{eqnarray}
which is equivalent to
\begin{eqnarray}
\label{eq:wigmr8}
   \mE e^{c_{3} \max_{\cX}  \lp \sum_{i=1}^n  \sum_{j=1}^m A_{ij} \x_i\x_j   \rp  } 
 \leq   e^{-\frac{c_{3}^2}{2}} \mE e^{c_{3}  \max_{\cX}  \sqrt{2} \lp\g^{(2)}\rp^T\x  } ,
\end{eqnarray}
and
\begin{eqnarray}
\label{eq:wigmr9}
     \log \lp \mE e^{c_{3} \max_{\cX}  \lp \sum_{i=1}^n \sum_{j=1}^m A_{ij}\x_i\x_j    \rp  }
   \rp
      \leq 
   -\frac{c_{3}^2}{2} + \log \lp
    \mE e^{c_{3}\sqrt{2}  \max_{\cX}   \lp\g^{(2)}\rp^T\x   }
    \rp .
\end{eqnarray}
Utilizing once again Chernoff/Markov inequality, we find
 \begin{eqnarray}
 \label{eq:wigldpeq2}
\frac{1}{n}\log \lp  \mP \lp \xi^+_{wig}
 \geq  u \rp \rp
& \leq  & \frac{1}{n} \log \lp \mE e^{c_3 \sqrt{n} \xi^+_{wig} -c_3\sqrt{n} u} \rp
=  \frac{1}{n} \log \lp \mE e^{c_3 \sqrt{n}  \xi^+_{wig} } \rp -\frac{1}{n}c_3\sqrt{n} u .
  \end{eqnarray}
A combination of (\ref{eq:wiginteq1ab0}), (\ref{eq:wigmr9}),  and (\ref{eq:wigldpeq2}) together with adopted scaling $c_3\rightarrow c_3\sqrt{n}$ also gives

 \begin{eqnarray}
 \label{eq:wigldpeq3}
 \frac{1}{n}\log \lp  \mP \lp \xi^+_{wig}
 \geq  u \rp \rp
& \leq  &    \frac{1}{n} \log \lp \mE e^{c_3 n \xi^+_{wig} } \rp - c_3 u 
\nonumber \\
& =  &    \frac{1}{n} \log \lp \mE e^{c_3 \sqrt{n}  \max_{\cX} \x^TA\x } \rp - c_3 u 
\nonumber \\
& \leq &
 -\frac{c_{3}^2}{2}  +  \frac{1}{n}  \log\lp  \mE_{\g^{(1)}} e^{c_{3}\sqrt{2}\sqrt{n} \max_{\cX}  \lp   \lp\g^{(1)}\rp^T\x   \rp } \rp
-c_3 u
\nonumber \\
& = &
  -\frac{c_{3}^2}{2}  +  \frac{1}{n}  \log\lp  \mE_{\g^{(1)}} e^{c_{3}\sqrt{2}\sqrt{n}\| \g^{(1)}\|_2} \rp
-c_3 u.
  \end{eqnarray}
Analogously to (\ref{eq:ldpeq4}) and (\ref{eq:ldpeq4a}), we write
 \begin{eqnarray}
 \label{eq:wigldpeq4}
 \lim_{n\rightarrow \infty} \frac{1}{n}  \log\lp  \mE_{\g^{(1)}} e^{c_{3}\sqrt{n} \| \g^{(1)}\|_2} \rp
=
\min_{\gamma_1 > 0 } \bar{\phi}_1(\gamma_1),
\end{eqnarray}
with
 \begin{eqnarray}
 \label{eq:wigldpeq4a}
\bar{\phi}_1(\gamma_1) = \gamma_1 c_3\sqrt{2} -\frac{1}{2} \log\lp 1- \frac{c_3\sqrt{2}}{2\gamma_1}\rp.
  \end{eqnarray}
 Equalling the $\gamma_1$ derivative to zero gives
 \begin{eqnarray}
 \label{eq:wigldpeq4aa0}
\frac{d\bar{\phi}_1(\gamma)}{d\gamma_1} = c_3\sqrt{2} - \frac{1}{ 2\gamma_1 - c_3\sqrt{2}}  + \frac{1}{2\gamma_1}
=c_3\sqrt{2} \lp 1 - \frac{1}{ 2\gamma_1 (2\gamma_1 - c_3\sqrt{2} ) } \rp = 0.
  \end{eqnarray}
One then finds the optimal $\gamma_1$
 \begin{eqnarray}
 \label{eq:wigldpeq5}
\hat{ \bar{\gamma} }_1 = \frac{c_3\sqrt{2}  +  \sqrt{2c_3^2  + 4} }{4}.
  \end{eqnarray}
 A combination of (\ref{eq:wigldpeq3})-(\ref{eq:wigldpeq5}) gives
 \begin{eqnarray}
 \label{eq:wigldpeq6}
 \frac{1}{n}\log \lp  \mP \lp \xi^+_{wig}
 \geq  u \rp \rp
& \leq  &  
   -\frac{c_{3}^2}{2}  +  \frac{1}{n}  \log\lp  \mE_{\g^{(1)}} e^{c_{3}\sqrt{2}\sqrt{n}\| \g^{(1)}\|_2} \rp
 -c_3 u
\nonumber \\
& = & 
   -\frac{c_{3}^2}{2}  + \bar{\phi}_1(\hat{\bar{\gamma}}_1) 
   -c_3 u ,
  \end{eqnarray}
which after optimization over $c_3$ gives (\ref{eq:wigthm1eq2}) and completes the theorem's proof.
\end{proof}

\subsection[]{Agreement with results from \cite{MajVer09}}
\label{sec:ldpwigagr}
 
 We consider the upper spectral edge, i.e., the maximal eigenvalue of $\bar{W}$  (due to symmetry, the lower spectral edge related to the minimal eigenvalue of $\bar{W}$ is identical).

\subsubsection{Upper spectral edge}
\label{sec:ldpwigagrmax}
 
 As stated earlier, in  \cite{MajVer09} the $\bar{W}$'s maximal eigenvalue LDP was derived via a Coulomb gas approach. The following is the explicit characterization obtained in \cite{MajVer09}. 
 
\begin{eqnarray}
\label{eq:wigagr1}
    \bar{\Phi}_u = \frac{u^2-1}{2} - \log\lp u\sqrt{2}\rp + \frac{1}{4 u^2} G\lp \frac{2}{u^2}\rp ,
  \end{eqnarray} 
where $G(z) =  _3{}{\mathcal F}_2([1,1,1.5],[2,3],z) $ is a hypergeometric function.
 
Following the presentation style of Section \ref{sec:ldpwishagr}, we in Figure \ref{fig:fig2} (the right portion) show in parallel both $\bar{\zeta}_u$ obtained via partially lifted RDT in Section \ref{sec:ldpmaxwig} and  $\bar{\Phi}_u$ obtained via Coulomb gas approach in \cite{MajVer09}  and given in (\ref{eq:wigagr1}). We consider the $u$ values outside the bulk of the spectrum and above the upper spectral edge, i.e., the $u$ values that satisfy $u\geq \mE \xi^+_{wig} = \sqrt{2} $. The obtained $\bar{\zeta}_u$ and $\bar{\Phi}_u$ curves are again visually identical.

  \begin{figure}[h]
\centering
\centerline{\includegraphics[width=.87\linewidth]{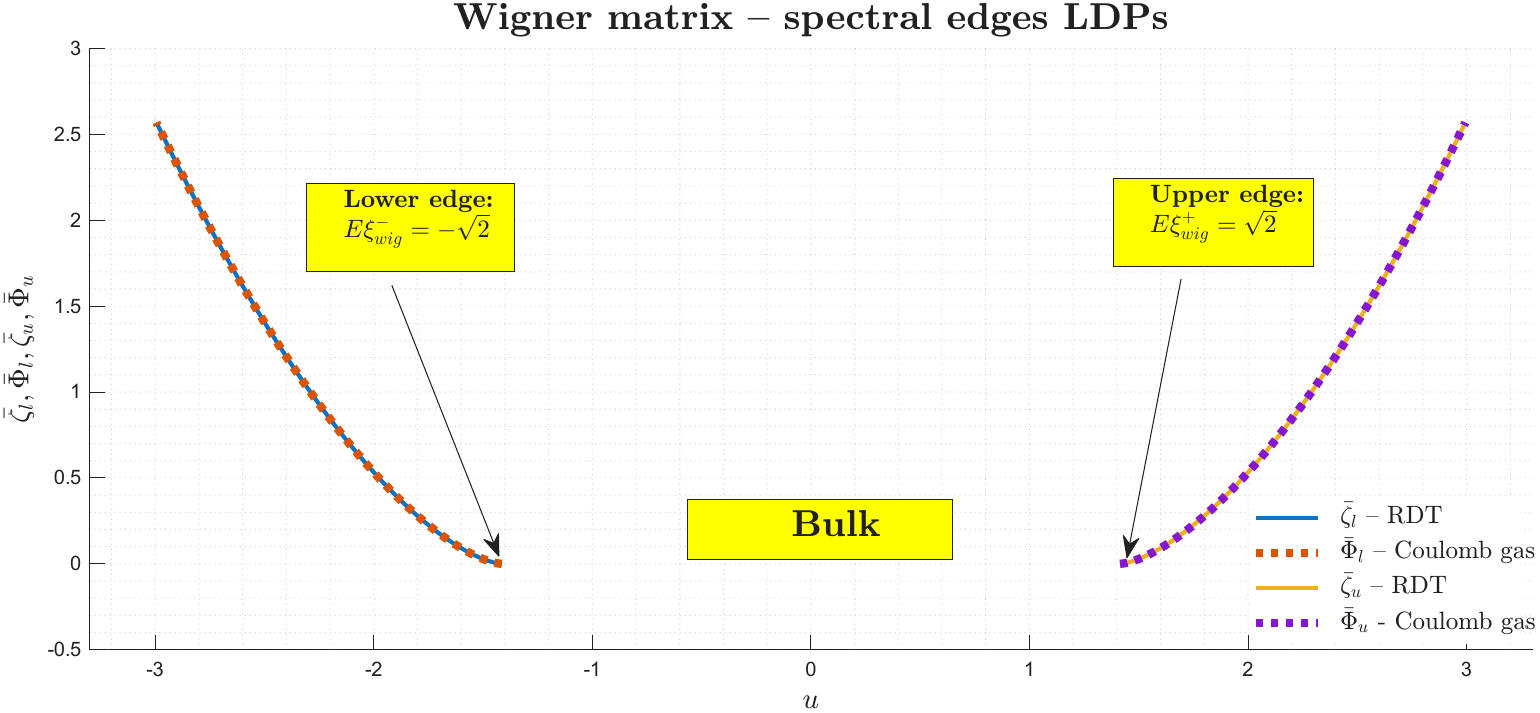}}
\caption{Wigner matrix -- lower ($\bar{\zeta}_l$, $\bar{\Phi}_l$) and upper ($\bar{\zeta}_u$, $\bar{\Phi}_u$) spectral edges LDPs; RDT ($\bar{\zeta}_l(u)=\bar{\zeta}_u(-u)$)  versus Coulomb gas ($\bar{\Phi}_l(u)=\bar{\Phi}_u(-u)$) }
\label{fig:fig2}
\end{figure}

To  show that the agreement between the two methodologies is not only visual but also fully numerical as well, we complement Figure \ref{fig:fig2} with Table \ref{tab:tab3}. For three concrete $u$ values ($u\in\{1.7342,2.2342,2.7342\}$), both $\bar{\zeta}_u$ and $\bar{\Phi}_u$ are given (due to the above mentioned symmetry of the standard normal random variables, $\bar{\zeta}_l(u)=\bar{\zeta}_u(-u)$ and $\bar{\Phi}_l(u)=\bar{\Phi}_u(-u)$; also, for completeness, relevant RDT parameters, $\hat{c}_3$ and $\hat{\bar{\gamma}}_1$, are provided in Table \ref{tab:tab3} as well). As Table \ref{tab:tab3} shows, one indeed has a full numerical agrement between $\bar{\zeta}_u$ and $\bar{\Phi}_u$.

\begin{table}[h]
\caption{Wigner upper (lower) spectral edge LDP rate function values obtained via RDT ($\bar{\zeta}_u$) and via Coulomb gas ($\bar{\Phi}_u$) }\vspace{.1in}
\centering
\def\arraystretch{1.2}
\begin{tabular}{ ||c||c||c||c|| }\hline\hline
 \hspace{-0in}  $u$                                             &  $1.7342$ &  $2.2342$ & $2.7342$   \\ \hline\hline
    $\hat{c}_3$                 &   $ 1.0037 $ &   $ 1.7296$   &   $ 2.3401 $   \\ \hline
$\hat{\gamma}_1$                                         &   $  0.9680$ &   $   1.4014$   &   $ 1.7940 $   \\ \hline \hline
$\bar{\zeta}_u(u)=\bar{\zeta}_l(-u)$                                         &  \bl{$\mathbf{0.2097}$} &  \bl{$\mathbf{  0.9016}$}   &  \bl{$\mathbf{ 1.9215 }$}   \\ \hline \hline
$\bar{\Phi}_u(u)=\bar{\Phi}_l(-u)$                                         &  $\mathbf{0.2097}$ &  $\mathbf{ 0.9016}$   &  $\mathbf{ 1.9215 }$   \\ 
\hline\hline
\end{tabular}
\label{tab:tab3}
\end{table}

\section{Conclusion}
\label{sec:conc}

We developed a novel methodology for studying random matrices' LDPs. The developed framework relies on a partially lifted variant of \emph{Random duality theory} (RDT) and completely circumvents prevalent traditional \emph{Coulomb gas} and \emph{tilted spherical integrals} random matrix theory approaches. The proposed methodology has two distinct features: (i) simplicity and (ii) generality. To demonstrate both of them, we apply it to Wishart and Wigner classical ensembles. For either of the considered ensembles, the obtained LDP characterizations of both upper and lower spectral edges are in a full agreement with the corresponding ones achieved through the traditional \emph{Coulomb gas} approaches in \cite{MajVer09,KatzPer10}.

Opportunities for further extensions are abundant. For example, they include almost all deformed and perturbed variants studied in recent literature and way beyond them. While conceptually similar to the core of the mechanism presented here, they also require additional technical refinements that are problem specific and will be discussed elsewhere.

%
%
%
%
%
%
%

\begin{singlespace}
\bibliographystyle{plain}
\bibliography{nflgscompyxRefs}
\end{singlespace}

\end{document}